\documentclass[a4paper,12pt, reqno]{amsart}
\usepackage[backref]{hyperref}
\usepackage{a4wide}
\usepackage{amsthm}
 \newtheorem{theorem}{Theorem}[section]
 
 \newtheorem{lemma}[theorem]{Lemma}

 \newcommand{\mbb}{\mathbb}
 \newcommand{\R}{\mbb{R}}

 \newcommand{\Z}{\mbb{Z}}

\begin{document}

\title{Diophantine approximation on lines with prime constraints}

\author{Stephan Baier}
\author{Anish Ghosh} 
\address{S. Baier, A. Ghosh, School of Mathematics, University of East Anglia, Norwich, NR4 7TJ, England}

\email{s.baier@uea.ac.uk, a.ghosh@uea.ac.uk}

\subjclass[2000]{11J83, 11K60, 11L07}

\begin{abstract}
We study the problem of Diophantine approximation on lines in $\mathbb{R}^2$ with prime numerator and denominator.
\end{abstract}
\maketitle


\section{Introduction}

\subsection*{Diophantine approximation with restricted numerator and denominator}

The study of Diophantine approximation with restrictions placed on the numerator and denominator has a long history. For instance, the problem of approximating irrational numbers with rationals whose denominators are prime was studied by Vinogradov, Heath-Brown and Jia \cite{HeJi} and others. The current best possible result is due to Matom\"{a}ki \cite{Mato} and states that for irrational $\alpha$, there exist infinitely many primes $p$ such that
\begin{equation}\label{eq:res}
|\langle \alpha p \rangle| < p^{-\nu}
\end{equation}
\noindent for any $\nu < 1/3$. Here $\langle~\rangle$ is the distance to the nearest integer. The corresponding problem with \emph{both} numerator and denominator prime was studied by Ramachandra and others \cite{Rama, Srini}. More generally, one can study analogues of Khintchine's theorem with restrictions. The case of prime denominator follows from the theorem of Duffin-Schaeffer \cite{DuSch}. Harman \cite{HarmK} has established the complete analogue of Khintchine's theorem in the case of both prime numerator and denominator. His work was extended to higher dimensions by Jones \cite{Jones}.

\subsection*{Metric Diophantine approximation on manifolds} The subject of \emph{metric Diophantine approximation on manifolds} is concerned with studying Diophantine properties of points lying on proper submanifolds of $\R^n$. The requirement of lying on a proper submanifold makes the theory significantly more complicated than traditional metric Diophantine approximation. Typically, one seeks to prove that points on suitable manifolds of $\R^n$ inherit the Diophantine behaviour of $\R^n$. For instance, say that a vector $(x_1,\dots, x_n)$ is \emph{very well approximable} if for some $\epsilon > 0$ there are infinitely many $q \in \Z$ such that
\begin{equation}
\max_{i}|qx_i + p_i| \leq |q|^{-(1+\epsilon)/n}
\end{equation}
\noindent for some $(p_1, p_2,\dots, p_n) \in \Z^n$. It is an easy consequence of the Borel-Cantelli lemma that Lebesgue almost every vector in $\R^n$ is not very well approximable. In $1932$, Mahler conjectured that almost every point on the curve $(x, x^2, \dots, x^n)$ is not very well approximable. This conjecture was settled by Sprind\v{z}uk, who in turn conjectured that almost every point on a ``nondegenerate" manifold is not very well approximable. We will not define ``nondegenerate" here, pointing the reader to \cite{KM} where this conjecture is settled in a stronger form using dynamics on the space of unimodular lattices. Informally,  a nondegenerate manifold is one which does not locally lie in an affine subspace. It is expected that nondegenerate manifolds inherit Diophantine behaviour from $\R^n$ and this has been checked in many important cases. 

At the other end of the spectrum lie affine subspaces. For these, one does not expect inheritance unless one specifies some Diophantine condition on the matrix defining the affine subspace. See \cite{Kl1, Ghosh} for examples of results in this direction.\\

It is natural to enquire if one can combine the two problems discussed above and develop a theory of metric Diophantine approximation on manifolds with restricted numerator and denominator. In \cite{HaJo}, Harman and Jones proved the following result on simultaneous approximation of $\alpha$ and $\alpha^{\tau}$ by fractions $r/p$ and $q/p$ with $p, q, r$ primes. 

\begin{theorem}(Harman-Jones) Let $\varepsilon>0$ and $\tau>1$. Then for almost all positive $\alpha$, with respect to the Lebesgue measure, there are infinitely many $p,q,r$ prime such that
\begin{equation}
0 < p\alpha - r< p^{-3/16+\varepsilon} \quad \mbox{and} \quad 0<p\alpha^{\tau} - q< p^{-3/16+\varepsilon}.
\end{equation}
\end{theorem}

Note that if $\tau > 1$ which is assumed in the theorem, the curve $(\alpha, \alpha^{\tau})$ is an example of a nondegenerate curve. The generic best possible exponent above is $-1/2$. In this paper, we investigate the question of  Diophantine approximation with restricted numerator and denominator for the simplest case of affine subspaces, namely lines in $\R^2$. More precisely, we study simultaneous approximation of $\alpha$ and $c\alpha$ by $r/p$ and $q/p$ with $r$ and $p$ primes, $q$ an integer and $c$ a fixed irrational number. We prove the following.

\begin{theorem} \label{super} Let $\varepsilon>0$ and let $c>1$ be an irrational number. Then for almost all positive $\alpha$, with respect to the Lebesgue measure, there are infinitely many triples $(p,q,r)$ with $p$ and $r$ prime and $q$ an integer such that
\begin{equation} \label{simultan}
0 < p\alpha - r < p^{-1/5+\varepsilon} \quad \mbox{and} \quad 0< pc\alpha - q < p^{-1/5+\varepsilon}.
\end{equation}
\end{theorem}

The above theorem is in fact valid more generally for lines of the form $(\alpha, c\alpha + d)$ as long as $c > 1$ is irrational, with identical proof. Our approach is broadly based on the method followed by Harman-Jones. However, we use exponential sum estimates rather than zero density estimates for the Riemann zeta function.\\

\noindent The proof of Theorem \ref{super} is carried out in \S's $2, 3, 4$ and $5$. In \S $6$, we discuss some variations of the problem and open questions.

\section{A Metrical approach}

Let $F_N(\alpha)$ be the number of solutions to \eqref{simultan} with $p<N$ and for $0<A<B$ let
$$
G_N(A,B)= \frac{A^2}{4B} \cdot N^{3/5+\varepsilon}
(\log N)^{-2}. 
$$

We will prove

\begin{theorem} \label{Theo}
There exists an infinite sequence $\mathcal{S}$ of natural numbers $N$ such that the following hold.\\ 

(i) Let $0 < A < B$. Then for all $a,b$ with $A\le a<b\le B$ we have
\begin{equation} \label{LOL}
\int\limits_a^b F_N(\alpha) d\alpha \ge (b-a)G_N(A,B)(1+o(1))
\end{equation}
if $N\in \mathcal{S}$ and $N\rightarrow \infty$.\\

(ii) Let $0<A<B$ and $\varepsilon>0$. Then there exists a constant $K=K(A,B,\varepsilon)$ such that, for $\alpha\in [A,B]$, we have 
$$
F_N(\alpha)\le KG_N(A,B)+J_N(\alpha)
$$
with
$$
\int\limits_A^B \left|J_N(\alpha)\right| d\alpha=o\left(G_N(A,B)\right) \quad \mbox{as } N\rightarrow \infty
$$
if $N\in \mathcal{S}$ and $N\rightarrow \infty$. 
\end{theorem}

Theorem \ref{Theo}(i) corresponds to Lemma 2, Theorem \ref{Theo}(ii) to Lemma 3 in \cite{HaJo}. The only difference is that $\mathcal{S}=\mathbb{N}$ in the said lemmas. However, it is sufficient to assume that \eqref{LOL} holds just for an infinite sequence of natural numbers $N$ rather than all $N$'s because we want to prove no more than the infinitude of triples $(p,q,r)$ satisfying \eqref{simultan}. Theorem \ref{Theo}, combined with Lemma $1$ in \cite{HaJo} will complete the proof of Theorem \ref{super}. 

\section{Proof of Theorem \ref{Theo}(i)}

\subsection{Construction of the sequence $\mathcal{S}$} \label{construction}
Let $q_1<q_2<q_3<...$ be the denominators in the continued fraction approximants for the irrational number $c$. We will set
$$
\mathcal{S}:=\{q_1^2,q_2^2,...\}.
$$
In the sequel, we shall suppose that $N\in \mathcal{S}$. 

\subsection{Reduction to a counting problem}
To prove Theorem \ref{Theo}(i), we broadly follow the approach in section 3 of \cite{HaJo} . However, we use exponential sum estimates instead of zero density estimates for the Riemann zeta function since they turn out to be more suitable for our purposes (this is the content of the next subsection).\\


Let 
$$
\mathcal{B}_p=\bigcup_{\substack{r\in \mathbb{P}\\ q\in \mathbb{N}}} \left[\left.\frac{r}{p},\frac{r+\eta}{p}\right)\right.\cap \left[\left. \frac{1}{c}\cdot \frac{q}{p},\frac{1}{c}\cdot \frac{q+\eta}{p}\right)\right.\cap [a,b],
$$
where $\eta=p^{\varepsilon/2-1/5}$. Then
\begin{equation} \label{integral}
\int\limits_{a}^{b} F_N(\alpha)d\alpha =\sum\limits_{\substack{p \in \mathbb{P}\\ p\le N}} \lambda(\mathcal{B}_p).
\end{equation}

\noindent Set
\begin{equation} \label{mudef}
\mu:=(a+b)/(2a).
\end{equation} 

\noindent Our strategy is to split the interval $[1,N]$ into subintervals $[P,P\mu]$ and sum up over the $P$'s in the end. Accordingly, we restrict $p$ to the interval $P\le p<P\mu$ with $P\mu\le N$. We then obtain a lower bound for \eqref{integral} by replacing $\eta$ with 
\begin{equation} \label{etadef}
\eta=(\mu P)^{-1/5+\varepsilon/2}.
\end{equation}

\noindent We note that if 
$$
\frac{r}{p}\le  \frac{1}{c}\cdot \frac{q}{p} \le \frac{r+\eta/2}{p},
$$
then 
$$
\lambda\left(\left[\left.\frac{r}{p},\frac{r+\eta}{p}\right)\right. \cap \left[\left.\frac{1}{c}\cdot \frac{q}{p},\frac{1}{c}\cdot \frac{q+\eta}{p}\right)\right.\right)\ge \nu,
$$
where
\begin{equation} \label{nudef}
\nu:=\frac{\eta}{\mu P}\min \left(\frac{1}{2}, \frac{1}{c}\right)=(\mu P)^{-6/5+\varepsilon/2}\min \left(\frac{1}{2}, \frac{1}{c}\right).
\end{equation}
Also, for all $p\in [P,P \mu)$,
$$
Pa\mu\le r\le bP \Longrightarrow a\le \frac{r}{p}\le b,
$$
and $r$ here runs over the primes in an interval of length $\frac{b-a}{2}P$. We thus have 
\begin{equation} \label{key}
\sum\limits_{P\le p< \mu P} \lambda(\mathcal{B}_p) \ge \nu N(P),
\end{equation}
where $N(P)$ counts the number of solutions $(q,p,r)\in \mathbb{Z}\times \mathbb{P} \times \mathbb{P}$ to 
$$
q\in \left[ \left. cr, cr+\delta\right.\right), \quad P\le p< P \mu, \quad Pa\mu\le r\le bP,  
$$
where
\begin{equation} \label{alphadef}
\delta:=\frac{c\eta}{2}=\frac{c}{2(\mu P)^{1/5-\varepsilon/2}}.
\end{equation}

\noindent Note that in contrast to the problem considered by Harman and Jones, the conditions on $p$ and  $q$ are here independent, which simplifies matters to some extent. By the prime number theorem, the number $R(P)$ of prime solutions to
$$
P\le p < P \mu, 
$$
satisfies 
\begin{equation} \label{RP}
R(P) \sim (\mu-1)P(\log P)^{-1}  \mbox{ as } P \rightarrow \infty.
\end{equation}
It remains to count the number of solutions $(q,r)\in \mathbb{N}\times \mathbb{P}$ to
$$
q\in \left[ \left. cr, cr+\delta\right.\right), \quad Pa\mu\le r\le bP,
$$ 
which equals
$$
S(P):=\sum\limits_{\substack{Pa\mu \le r \le bP\\ r \ prime}} \left([-cr]-[-(cr+\delta)]\right).
$$
Let
$$
T(P):=\sum\limits_{Pa\mu \le n \le bP} \left([-cn]-[-(cn+\delta)]\right)\Lambda(n).
$$
We aim to show that 
\begin{equation} \label{TP}
T(P) = \delta(b-a\mu)P(1+o(1)) + O\left(N^{4/5+\varepsilon/3}\right)  \quad \mbox{if } P\mu\le N.
\end{equation}
As usual, from \eqref{TP}, it follows that
$$
S(P) = \delta(b-a\mu)P (\log P)^{-1}(1+o(1)) + O\left(N^{4/5+\varepsilon/3}\right) \quad \mbox{if } P\mu\le N,
$$
which together with \eqref{RP}  gives
$$
N(P) = R(P)S(P)= \delta (b-a\mu)(\mu-1) P^{2} (\log P)^{-2}(1+o(1))+ O\left(N^{9/5+\varepsilon/3}\right) \quad \mbox{if } P\mu\le N.
$$
Combing this with \eqref{mudef}, \eqref{etadef}, \eqref{nudef}, \eqref{key} and \eqref{alphadef}, we obtain
\begin{equation} \label{near}
\begin{split}
\sum\limits_{P\le p< \mu P} \lambda(\mathcal{B}_p)\ge & \frac{(b-a)^2}{4a}\cdot \min \left(\frac{c}{4},\frac{1}{2}\right) \cdot (\mu P)^{-7/5+\varepsilon} P^2 (\log P)^{-2}(1+o(1))+\\
& O\left(\delta N^{3/5+5\varepsilon/6}\right) \quad \mbox{if } P\mu\le N.
\end{split}
\end{equation}

By splitting the interval $[1,N)$ into intervals of the form $[P,\mu P)$ and summing up, it now follows from \eqref{integral} and \eqref{near} that
\begin{equation*}
\begin{split}
\int\limits_{a}^{b} F_N(\alpha)d\alpha \ge & \frac{(b-a)^2}{4a}\cdot \min \left(\frac{c}{4},\frac{1}{2}\right) \cdot \left( \sum\limits_{k=0}^{\infty} 
\left(\frac{N}{\mu^k}\right)^{-7/5+\varepsilon} \left(\frac{N}{\mu^{k+1}}\right)^2 (\log N)^{-2}\right)(1+o(1))\\
= & \frac{(b-a)^2}{4a}\cdot \min \left(\frac{c}{4},\frac{1}{2}\right) \cdot \mu^{-2} \cdot \frac{1}{1-\mu^{-(3/5+\varepsilon)}} \cdot 
N^{3/5+\varepsilon} (\log N)^{-2}(1+o(1)).
\end{split}
\end{equation*}
Further, since $\mu>1$, we have 
$$
1-\mu^{-(3/5+\varepsilon)}\le \left(\frac{3}{5}+\varepsilon\right)(\mu-1)= \left(\frac{3}{5}+\varepsilon\right)\cdot \frac{b-a}{2a}.
$$
Hence, we deduce that 
\begin{equation}
\begin{split}
\int\limits_{a}^{b} F_N(\alpha)d\alpha \ge & (b-a) \cdot \frac{a^2}{a+b} \cdot \frac{1}{2\left(\frac{3}{5}+\varepsilon\right)}\cdot \min\left(c,2\right) \cdot N^{3/5+\varepsilon}
(\log N)^{-2}(1+o(1))\\ \ge & 
(b-a) \cdot \frac{a^2}{4b} \cdot N^{3/5+\varepsilon} \ge (b-a) \cdot \frac{A^2}{4B} \cdot N^{3/5+\varepsilon}
(\log N)^{-2}(1+o(1)), 
\end{split}
\end{equation}
provided that $\varepsilon\le 2/5$, establishing the claim of Theorem \ref{Theo}(i).  It remains to prove \eqref{TP}.

\section{Reduction to exponential sums}
For $x\in \mathbb{R}$ let
$$
\psi(x):=x-[x]-\frac{1}{2}.
$$
Then we may write $T(P)$ in the form
$$
T(P)=\delta \sum\limits_{Pa\mu \le n \le bP} \Lambda(n) + \sum\limits_{Pa\mu \le n \le bP} (\psi(-cn)-\psi(-(cn+\delta)))\Lambda(n).
$$
By the prime number theorem,
$$
\delta \sum\limits_{Pa\mu \le n \le bP} \Lambda(n) \sim \delta(b-a\mu)P \quad \mbox{as } P\rightarrow \infty.
$$
Hence, to establish \eqref{TP}, it suffices to prove that for any fixed $\varepsilon>0$ a bound of the form
\begin{equation} \label{error}
\sum\limits_{Pa\mu \le n \le bP} (\psi(-cn)-\psi(-(cn+\delta)))\Lambda(n)=O\left(N^{4/5+\varepsilon}\right) \quad \mbox{if } P\mu\le N
\end{equation}
holds. We reduce the left-hand to exponential sums, using the following Fourier analytic tool developed by Vaaler \cite{Vaal}.

\begin{lemma}[Vaaler] \label{Vaaler} For $0<|t|<1$ let
$$
W(t)=\pi t(1-|t|) \cot \pi t +|t|.
$$
Fix a positive integer $J$. For $x\in \mathbb{R}$ define 
$$
\psi^{\ast}(x):=-\sum\limits_{1\le |j|\le J} (2\pi i j)^{-1}W\left(\frac{j}{J+1}\right)e(jx)
$$
and
$$
\delta(x):=\frac{1}{2J+2} \sum\limits_{|j|\le J} \left(1-\frac{|j|}{J+1}\right)e(jx).
$$
Then $\delta(x)$ is non-negative, and we have 
$$
|\psi^{\ast}(x)-\psi(x)|\le \delta(x)
$$
for all real numbers $x$. 
\end{lemma}

\begin{proof}
This is Theorem A6 in \cite{GrKo} and has its origin in \cite{Vaal}.
\end{proof}

It follows that
\begin{equation} \label{S123}
\sum\limits_{Pa\mu \le n \le bP} (\psi(-cn)-\psi(-(cn+\delta)))\Lambda(n)=S_1+O(S_2+S_3)
\end{equation}
for any positive integer $J$, where
$$
S_1:=-\sum\limits_{1\le |j|\le J} (2\pi i j)^{-1}W\left(\frac{j}{J+1}\right)(1-e(j\delta))\sum\limits_{Pa\mu \le n \le bP} e(-jcn)\Lambda(n),
$$
$$
S_2:=\frac{1}{2J+2} \sum\limits_{1\le |j|\le J} \left(1-\frac{|j|}{J+1}\right)(1-e(j\delta))\sum\limits_{Pa\mu \le n \le bP} e(-jcn)\Lambda(n)
$$
and
$$
S_3:=\frac{1}{2J+2} \cdot (1-e(j\delta))\sum\limits_{Pa\mu \le n \le bP} \Lambda(n).
$$
We choose
$$
J:=P^{1/5}.
$$
Then
\begin{equation} \label{S3}
S_3\ll P^{4/5}\ll N^{4/5}. 
\end{equation}

We treat $S_1$ and $S_2$ at once by observing that
\begin{equation} \label{S12}
|S_1|+|S_2|\ll \sum\limits_{1\le j\le J} \frac{1}{j} \cdot \left| \sum\limits_{Pa\mu \le n \le bP} e(jcn)\Lambda(n)\right| \ll 
(\log P) \max\limits_{1\le H\le J} \frac{1}{H} \cdot Z(H),
\end{equation}
where
\begin{equation} \label{ZH}
Z(H):=\sum\limits_{H\le h\le 2H} \left| \sum\limits_{Pa\mu \le n \le bP} e(hcn)\Lambda(n)\right|. 
\end{equation}
Now it remains to estimate $Z(H)$.

\subsection{Application of Vaughan's identity}
We further convert the inner sum involving the von Mangoldt function on the right-hand side of \eqref{ZH}  into bilinear sums using Vaughan's identity.

\begin{lemma}[Vaughan] \label{Vaughan}
Let $U\ge 1$, $V\ge 1$, $UV\le x$. Then we have for every arithmetic function $f: \mathbb{N}\rightarrow \mathbb{C}$ the estimate
$$
\sum\limits_{U<n\le x} f(n)\Lambda(n) \ll (\log 2x)T_1+T_2
$$
with
$$
T_1:= \sum\limits_{l\le UV} \max\limits_{w} \left| \sum\limits_{w\le k\le x/l} f(kl) \right|,
$$
$$
T_2:=\left| \sum\limits_{U<m\le x/V} \sum\limits_{V<k\le x/m} \Lambda(m)b(k) f(mk) \right|,
$$
where $b(k)$ is an arithmetic function which only depends on $V$ and satisfies the inequality $b(k)\le \tau(k)$, $\tau(k)$ being the number of divisors of $k$.
\end{lemma}

\begin{proof}
This is Satz 6.1.2. in \cite{Brud} and has its origin in \cite{Vaug}.
\end{proof}

We use Lemma \ref{Vaughan} with parameters $U$ and $V$ satisfying $1\le U=V\le (Pa\mu)^{1/2}$, to be fixed later, $x:=bP$ and 
$$
f(n):=\begin{cases} e(hcn) & \mbox{ if } Pa\mu \le n\le bP, \\ 0 & \mbox{ if } n<Pa\mu \end{cases}
$$
to deduce that
\begin{equation} \label{ZHest}
Z(H)\ll (\log P)Z_1(H)+Z_2(H),
\end{equation}
where
$$
Z_1(H):=\sum\limits_{H\le h\le 2H} \sum\limits_{l\le U^2} \max\limits_{Pa\mu/l \le w\le bP/l} \left| \sum\limits_{w\le k\le bP/l} e(hckl) \right|
$$
and 
$$
Z_2(H):=\sum\limits_{H\le h\le 2H} \left| \sum\limits_{U<m\le bP/U} \sum\limits_{Pa\mu/m\le k\le bP/m} \Lambda(m)b(k) e(hcmk) \right|
$$

Obviously,
\begin{equation} \label{Z1H}
\begin{split}
Z_1(H)\ll & \sum\limits_{H\le h\le 2H} \sum\limits_{l\le U^2} \min\left(\frac{P}{l}, ||\alpha hl||^{-1}\right) \le \sum\limits_{r\le 2HU^2} \tau(r) \min\left(\frac{HP}{r}, ||\alpha r||^{-1}\right)\\
\ll & P^{\varepsilon} \sum\limits_{r\le 2HU^2} \min\left(\frac{HP}{r}, ||\alpha r||^{-1}\right).
\end{split}
\end{equation}
We shall boil down $Z_2(H)$ to similar terms. Rearranging the summation gives 
$$
Z_2(H)=\sum\limits_{H\le h\le 2H} \left| \sum\limits_{a\mu U/b\le k\le bP/U} b(k) \sum\limits_{M_1(k) <m\le M_2(k)}  \Lambda(m) e(hcmk) \right|
$$
with 
$$
M_1(k):=\max([U]+1,Pa\mu/k) \quad \mbox{and} \quad M_2(k):=\min(bP/U,bP/k).
$$
We now observe that
\begin{equation} \label{Z2HK}
Z_2(H)\ll (\log P) \max\limits_{U\le K\le bP/U} Z(H,K),
\end{equation}
where
$$
Z(H,K) := \sum\limits_{H\le h\le 2H} \sum\limits_{K\le k\le 2K} b(k) \cdot \left| \sum\limits_{\max([U]+1,Pa\mu/k)<m\le \min(bP/U,bP/k)} \Lambda(m) e(hcmk) \right|. 
$$
Using the Cauchy-Schwarz inequality and the well-known bound
$$
\sum\limits_{K\le k\le 2K} |b(k)|^2 \ll K(\log 2K)^3 
$$ 
and expanding the square, we obtain 
\begin{equation} \label{aftercauchy}
\begin{split}
& Z(H,K)^2 \\ \ll
& H \left(\sum\limits_{K\le m\le 2K} |b(k)|^2\right)  
\sum\limits_{H\le h\le 2H} \sum\limits_{K\le k\le 2K} \left| \sum\limits_{M_1(k)\le m\le M_2(k)} \Lambda(m) e(hcmk) \right|^2\\
\ll & HK(\log 2K)^3 \sum\limits_{H\le h\le 2H} \sum\limits_{K\le k\le 2K}  \sum\limits_{M_1(k)\le m\le M_2(k)} \Lambda(m)^2 +\\
& HK(\log 2K)^3 \sum\limits_{H\le h\le 2H} \left| \sum\limits_{K\le k\le 2K}  \sum\limits_{M_1(k)\le m_1<m_2\le M_2(k)} \Lambda(m_1)\Lambda(m_2)e\left(hc(m_1-m_2)k\right)\right|\\
\ll & H^2KP(\log 2P)^5+ HK(\log 2P)^5\times\\ &  \sum\limits_{H\le h\le 2H} \sum\limits_{U\le m_1<m_2\le bP/K} \left| \sum\limits_{\max(K,Pa\mu/m_1)\le k\le \min(2K,bP/m_2)} e\left(hc(m_1-m_2)k\right)\right|\\
\ll & H^2KP(\log 2P)^5+ HK(\log 2P)^5\times\\ &  \sum\limits_{H\le h\le 2H} \sum\limits_{U\le m_1<m_2\le bP/K} \min\left(K,||hc(m_1-m_2)||^{-1}\right)\\
\ll & H^2KP(\log 2P)^5+ HP(\log 2P)^5 \sum\limits_{r\le 2HbP/K} \tau(r) \min\left(K,||rc|^{-1}\right)\\
\ll & H^2KP(\log 2P)^5+ HP^{1+\varepsilon} \sum\limits_{r\le 2HbP/K} \min\left(\frac{HP}{r},||rc|^{-1}\right).
\end{split}
\end{equation}

\subsection{Completion of the proof}
In both the treatments of $Z_1(H)$ and $Z_2(H)$, sums of the form 
$$
R_c(L,x):=\sum\limits_{l\le L} \min\left(\frac{x}{l},||rc||^{-1}\right).
$$
Their estimation is standard and depends on the Diophantine properties of $c$. We quote the following well-known result.

\begin{lemma} \label{standard} Let $L\ge 1$ and $x>1$. Suppose that $|c-a/q|\le q^{-2}$ with $a\in \mathbb{Z}$, $q\in \mathbb{N}$ and $(a,q)=1$. Then
$$
R_c(L,x)\ll \left(\frac{x}{q}+L+q\right)(\log 2Lqx).
$$
\end{lemma}

\begin{proof} This is Lemma 6.4.4. in \cite{Brud}
\end{proof} 

Recall the construction of the set $\mathcal{S}$ in subsection \ref{construction}. Since $N\in \mathcal{S}$, we have $q=\sqrt{N}\in \mathbb{N}$, and $q$ is the denominator of a continued fraction approximant for $c$, and thus $|c-a/q|\le q^{-2}$, as required in Lemma \ref{standard}. Now, combing \eqref{aftercauchy}, Lemma \ref{standard} and taking into account that $P\ll N$, $H\le J=P^{1/5}\ll N^{1/5}$ and $q=\sqrt{N}$, we obtain 
$$
Z(H,K)\ll N^{\varepsilon}\left(H(PK)^{1/2}+HPN^{-1/4}+HPK^{-1/2}+(HP)^{1/2}N^{1/4}\right), 
$$
which together with \eqref{Z2HK} gives
\begin{equation} \label{Z2Hendest}
Z_2(H)\ll N^{\varepsilon}\left(HPN^{-1/4}+HPU^{-1/2}+(HP)^{1/2}N^{1/4}\right).
\end{equation}
Similarly, \eqref{Z1H}, Lemma \ref{standard} and $q=\sqrt{N}$ imply
\begin{equation} \label{Z1Hendest}
Z_2(H)\ll N^{\varepsilon}\left(HPN^{-1/2}+U^2+N^{1/2}\right).
\end{equation}
Combining \eqref{ZHest}, \eqref{Z2Hendest} and \eqref{Z1Hendest}, taking into account that $P\ll N$ and choosing $U:=P^{2/5}$, we get
$$
Z(H)\ll HN^{4/5+\varepsilon},
$$
which together with \eqref{S12} gives
\begin{equation} \label{S12esti}
|S_1|+|S_2|\ll N^{4/5+\varepsilon}.
\end{equation}
Finally, combining \eqref{S123}, \eqref{S3} and \eqref{S12esti}, we deduce that \eqref{error} holds, which completes the proof of Theorem \ref{Theo}(ii).

\section{Proof of Theorem \ref{Theo}(ii)} 

\subsection{Sieve theoretical approach}
We are broadly following the treatment of \cite{HaJo} with appropriate modifications because the linear case, considered here, requires a different treatment. In particular, as in the previous section, the Diophantine properties of $c$ will come into play. Let $\{\cdot\}$ represent the fractional part, and put
$$
\mu:=N^{\varepsilon-1/5}.
$$
Write
$$
\mathcal{A}=\mathcal{A}(\alpha)=\{n[n\alpha][nc\alpha] \ :\  1\le n\le N, \ \{n\alpha\}<\mu, \ \{nc\alpha\}<\mu\}.
$$
We desire to show that $\mathcal{A}$ does not contain too many products of three primes. To this end, we apply a three-dimensional upper bound sieve (see \cite{HaRi}, Theorem 5.2). We therefore need to obtain an asymptotic formula for the number of solutions to 
$$
n[n\alpha][nc\alpha]\equiv 0 \bmod{d}, \ 1\le n\le N,
$$
with 
\begin{equation} \label{max}
\max(\{n\alpha\},\{nc\alpha\})<\mu,
\end{equation}
for 
$$
d\le D:=N^{\varepsilon}.
$$
For this it suffices to establish a formula for the number of solutions, say $S(\alpha;d_1,d_2,d_3)$, to
$$
n \equiv 0 \bmod{d_1}, \quad [n\alpha]\equiv 0 \bmod{d_2}, \quad [nc\alpha]\equiv 0 \bmod{d_2}, \quad 1\le n\le N
$$
subject to \eqref{max}. We can combine \eqref{max} with the congruence conditions to require
\begin{equation} \label{threeconds}
1\le n\le \frac{N}{d_1}, \quad \left\{\frac{nd_1\alpha}{d_2}\right\}<\frac{\mu}{d_2}, \quad \left\{\frac{nd_1c\alpha}{d_3}\right\}<\frac{\mu}{d_3}.
\end{equation}
We can count the number of solutions to \eqref{threeconds} using Fourier analysis. Let $L=D^4\mu^{-1}$. By the method of section 4 in \cite{Harm} we have
$$
S(\alpha;d_1,d_2,d_3)=\frac{N\mu^2}{d_1d_2d_3}+O\left(\frac{N\mu^2}{D^4}+E(\alpha;d_1,d_2,d_3)\right),
$$
where
$$
E(\alpha;d_1,d_2,d_3):=\frac{\mu^2}{d_2d_3}\sum\limits_{\substack{-L\le m_j\le L\\ (m_1,m_2)\not=(0,0)}} \left| \sum\limits_{1\le n\le N}
e\left(nd_1\left(\frac{\alpha m_1}{d_2}+\frac{c\alpha m_2}{d_3}\right)\right)\right|.
$$
Now, applying the upper bound sieve gives
\begin{equation}
F_N(\alpha)\le \frac{C(\varepsilon)\mu^2N}{\log^3 N} + O\left(J_N(\alpha)\right),
\end{equation}
where 
\begin{equation*}
\begin{split}
J_N(\alpha):= & \sum\limits_{d_1d_2d_3\le D} (d_1d_2d_3)^{\varepsilon}\left(\frac{N\mu^2}{D^4}+E(\alpha;d_1,d_2,d_3)\right)\\ = & \sum\limits_{d_1d_2d_3\le D} (d_1d_2d_3)^{\varepsilon}E(\alpha;d_1,d_2,d_3)+o\left(\frac{\mu^2N}{\log^3 N}\right) \quad \mbox{as } N\rightarrow \infty.
\end{split}
\end{equation*}
Hence, to establish the claim in Theorem \ref{Theo}(i), it suffices to show that
\begin{equation} \label{average}
\sum\limits_{d_1d_2d_3\le D} (d_1d_2d_3)^{\varepsilon}\int\limits_A^B E(\alpha;d_1,d_2,d_3)d\alpha = o\left(\frac{\mu^2N}{\log^3 N}\right) \quad \mbox{as } N\rightarrow \infty, \ N\in \mathcal{S}.
\end{equation}

\subsection{Average estimation for $E(\alpha;d_1,d_2,d_3)$} 
To estimate the integral in \eqref{average}, we observe that
\begin{equation} \label{E2alpha}
E(\alpha;d_1,d_2,d_3)\ll \frac{\mu^2}{d_2d_3} \sum\limits_{\substack{-L\le m_j\le L\\ (m_1,m_2)\not=(0,0)}} \min\left(\frac{N}{d_1},\frac{1}{||\alpha(m_1d_1/d_2+cm_2d_1/d_3)||}\right)
\end{equation}
and use the following lemma.

\begin{lemma} \label{intlemma}
Suppose that $0<A<B$, $K\ge 2$ and $x\not=0$. Then
\begin{equation} \label{intesti}
\int\limits_A^B \min\left(K,\frac{1}{||\alpha x||}\right)d\alpha=O_{A,B}\left(\min\left\{K,\max\left\{1,|x|^{-1}\right\}\log K\right\}\right).
\end{equation}
\end{lemma}

\begin{proof} We confine ourselves to the case when $x>0$ since the case when $x<0$ is similar. By change of variables $\beta=\alpha x$, we get
\begin{equation} \label{change}
\int\limits_A^B \min\left(K,\frac{1}{||\alpha x||}\right)d\alpha=\frac{1}{x}\int\limits_{xA}^{xB}  \min\left(K,\frac{1}{||\beta||}\right)d\beta.
\end{equation}
By periodicity of the integrand, if $x(B-A)\ge 1$, we have
\begin{equation} \label{case1}
\frac{1}{x} \int\limits_{xA}^{xB}  \min\left(K,\frac{1}{||\beta||}\right)d\beta\le \frac{[x(A-B)+1]}{x} \int\limits_0^1 \min\left(K,\frac{1}{||\beta||}\right)d\beta =O_{A,B}\left(\log K\right).
\end{equation}
If $1/K\le x(B-A)<1$, then
\begin{equation} \label{case2}
\frac{1}{x} \int\limits_{xA}^{xB}  \min\left(K,\frac{1}{||\beta||}\right)d\beta=\frac{1}{x}\int\limits_0^{x(B-A)} \min\left(K,\frac{1}{||\beta||}\right)d\beta =O_{A,B} \left(\frac{\log K}{x}\right). 
\end{equation}
If $0<x(B-A)< 1/K$, then trivially
\begin{equation} \label{case3}
\frac{1}{x}\int\limits_{xA}^{xB} \min\left(K,\frac{1}{||\alpha x||}\right)d\alpha=O_{A,B}(K).
\end{equation}
Combining \eqref{change}, \eqref{case1} and \eqref{case2}, we deduce the claim when $x>0$, which completes the proof. 
\end{proof}

Note that if $(m_1,m_2)\not=(0,0)$, then the term $m_1d_1/d_2+cm_2d_1/d_3$ on the right-hand side of \eqref{intesti} is non-zero because $c$ is supposed to be irrational. Now, employing \eqref{E2alpha} and Lemma \eqref{intlemma}, we get
\begin{equation} \label{E2using}
\begin{split}
& \int\limits_A^B E(\alpha;d_1,d_2,d_3)d\alpha\\ \ll & \frac{\mu^2}{d_2d_3} \left(\sum\limits_{\substack{-L\le m_2\le L\\ m_2\not=0}} \sum\limits_{-L\le m_1\le L}  \min\left\{\frac{N}{d_1},\max\left\{1,\frac{1}{d_1}\cdot \left|\frac{m_1}{d_2}+\frac{cm_2}{d_3}\right|^{-1}\right\}\log N\right\}\right.+\\
& \left. \sum\limits_{\substack{-L\le m_1\le L\\ m_1\not=0}}  \min\left\{\frac{N}{d_1},\max\left\{1,\frac{d_2}{m_1d_1}\right\}\log N\right\}\right)\\
\ll & \frac{\mu^2}{d_1d_2d_3} \sum\limits_{\substack{-L\le m_2\le L\\ m_2\not=0}} \min\left\{N,R(m_2)^{-1}\log N\right\}+ \frac{\mu^2L^2\log N}{d_2d_3}.
\end{split}
\end{equation}
where 
$$
R(m_2):=\min\limits_{m\in \mathbb{Z}} \left| \frac{m}{d_2}+\frac{cm_2}{d_3}\right|.
$$
Now we use the Diophantine properties of $c$. Since $N\in \mathcal{S}$, we have
$$
c=\frac{a}{\sqrt{N}}+\frac{\beta}{N} \quad \mbox{with } a\in \mathbb{Z}, \ \sqrt{N}\in \mathbb{N}, \ (a,\sqrt{N})=1, \  0\le |\beta|\le 1.
$$
It follows that
$$
R(m_2)\ge \frac{1}{d_2d_3\sqrt{N}}-\frac{cL}{N}\ge \frac{1}{2d_2d_3\sqrt{N}} \quad \mbox{if } 0<|m_2|\le L \mbox{ and $N$ is large enough.} 
$$
This together with \eqref{E2using} gives 
$$
 \int\limits_A^B E(\alpha;d_1,d_2,d_3)d\alpha \ll \frac{\mu^2L\sqrt{N}\log N}{d_1} \quad \mbox{if } 1\le d_2,d_3\le D,
$$
which implies \eqref{average} and thereby completes the proof of Theorem \ref{Theo}(ii).\\ 


\section{Remarks and open questions}

We end with remarks on some variations of Theorem \ref{super} and some open questions. 

\begin{enumerate}
\item An interesting problem would be to prove Khintchine type theorems with restrictions on manifolds. Let $\psi : \mathbb{R}_{+} \to \mathbb{R}_{+} \cup \{0\}$. The question would be to find necessary and sufficient conditions on $\psi$ and on $c$ which guarantee that for almost every $\alpha$,  
\begin{equation} \label{simultan2}
0 < p\alpha - r < \psi(p) \quad \mbox{and} \quad 0< pc\alpha - q < \psi(p)
\end{equation}
\noindent has infinitely many solutions in primes $p, r$ and integers $q$. In the setting where there are no restrictions on rationals, the conditions needed on $c$ are stronger than irrationality cf. \cite{Kl1, Ghosh}.\\ 
\item In Theorem \ref{super} there is an interplay between conditions imposed on $c$ and the quality of the exponent obtained. It is possible that a better exponent and even an asymptotic count for the number of solutions is possible under stronger conditions on $c$.\\
\item A stronger non-asymptotic result can probably be obtained by a lower bound sieve and under GRH. Moreover, one can consider the same question as in \cite{HaJo} but with the weaker condition that $q$ is an integer instead of being a prime. These along with higher dimensional versions of Theorem \ref{super} will be investigated elsewhere.

\end{enumerate}

\end{document}